\def\DateTime{14/September/2026}
\def\yes{\if00}
\def\no{\if01}
\def\iftenpt{\no}
\def\ifelevenpt{\no}
\def\iftwelvept{\yes}

\def\ifquery{\yes}

\iftenpt
\documentclass[leqno]{amsart}
\else\fi
\ifelevenpt
\documentclass[leqno,11pt]{amsart}
\else\fi
\iftwelvept
\documentclass[leqno,12pt]{amsart}
\else\fi

\usepackage{latexsym,amscd,color}
\usepackage{amsmath,amsfonts,amssymb,mathrsfs,mathtools}
\usepackage{qsymbols,euscript,enumitem}
\usepackage{url,xspace}
\usepackage{verbatim}
\usepackage{tikz}
\usepackage{tikz-cd}
\usetikzlibrary{positioning,arrows,shapes}
\usepackage[normalem]{ulem}

\usepackage{newtxtext,newtxmath}
\usepackage[T1]{fontenc}

\ifelevenpt
\else\fi

\iftwelvept
\else\fi

\theoremstyle{plain}
\newtheorem{theorem}{Theorem}[section]
\newtheorem{proposition}[theorem]{Proposition}
\newtheorem{lemma}[theorem]{Lemma}

\newtheorem{corollary}[theorem]{Corollary}

\newtheorem{claim}{Claim}[theorem]

\theoremstyle{definition}

\newtheorem{remark}[theorem]{Remark}

\newtheorem{example}[theorem]{Example}
\newtheorem{conjecture}[theorem]{Conjecture}

\numberwithin{equation}{section}

\newcommand{\ZZ}{{\mathbb{Z}}}
\newcommand{\QQ}{{\mathbb{Q}}}
\newcommand{\RR}{{\mathbb{R}}}

\newcommand{\PP}{{\mathbb{P}}}

\newcommand{\OO}{{\mathcal{O}}}

\newcommand{\Proj}{\operatorname{Proj}}

\newcommand{\Spec}{\operatorname{Spec}}

\newcommand{\ord}{\operatorname{ord}}

\newcommand{\adeg}{\widehat{\operatorname{deg}}}

\newcommand{\GCD}{\operatorname{GCD}}

\ifquery
\def\query#1{\setlength\marginparwidth{65pt} 
\marginpar{\raggedright\fontsize{7.81}{9} 
\selectfont\upshape\hrule\smallskip 
#1\par\smallskip\hrule}} 

\else
\def\query#1{}
\fi

\newcommand{\ndot}{\raisebox{.4ex}{.}}

\def\colorsout#1{\bgroup\markoverwith{\textcolor{#1}{\rule[0.5ex]{2pt}{0.7pt}}}\ULon} 
\def\coloruline#1{\bgroup\markoverwith{\textcolor{#1}{\rule[-0.5ex]{2pt}{0.7pt}}}\ULon} 

\definecolor{rose}{rgb}{1.0, 0.0, 0.5}
\definecolor{rosewood}{rgb}{0.4, 0.0, 0.04}
\definecolor{darkspringgreen}{rgb}{0.09, 0.45, 0.27}
\definecolor{mgreen}{rgb}{0.09, 0.45, 0.27}
\definecolor{orange-red}{rgb}{1.0, 0.27, 0.0}
\definecolor{mred}{rgb}{0.83, 0.0, 0.0}

\DeclareSymbolFont{bbold}{U}{bbold}{m}{n}
\DeclareMathSymbol{\bbalpha}{\mathord}{bbold}{"0B}
\DeclareMathSymbol{\bbbeta}{\mathord}{bbold}{"0C}
\DeclareMathSymbol{\bbgamma}{\mathord}{bbold}{"0D}
\DeclareMathSymbol{\bbdelta}{\mathord}{bbold}{"0E}
\DeclareMathSymbol{\bbespilon}{\mathord}{bbold}{"0F}
\DeclareMathSymbol{\bbzeta}{\mathord}{bbold}{"10}
\DeclareMathSymbol{\bbeta}{\mathord}{bbold}{"11}
\DeclareMathSymbol{\bbtheta}{\mathord}{bbold}{"12}
\DeclareMathSymbol{\bbiota}{\mathord}{bbold}{"13}
\DeclareMathSymbol{\bbkappa}{\mathord}{bbold}{"14}
\DeclareMathSymbol{\bblambda}{\mathord}{bbold}{"15}
\DeclareMathSymbol{\bbmu}{\mathord}{bbold}{"16}
\DeclareMathSymbol{\bbnu}{\mathord}{bbold}{"17}
\DeclareMathSymbol{\bbxi}{\mathord}{bbold}{"18}
\DeclareMathSymbol{\bbpi}{\mathord}{bbold}{"19}
\DeclareMathSymbol{\bbrho}{\mathord}{bbold}{"1A}
\DeclareMathSymbol{\bbsigma}{\mathord}{bbold}{"1B}
\DeclareMathSymbol{\bbtau}{\mathord}{bbold}{"1C}
\DeclareMathSymbol{\bbupsilon}{\mathord}{bbold}{"1D}
\DeclareMathSymbol{\bbphi}{\mathord}{bbold}{"1E}
\DeclareMathSymbol{\bbchi}{\mathord}{bbold}{"1F}
\DeclareMathSymbol{\bbpsi}{\mathord}{bbold}{"20}

\author{Atsushi MORIWAKI}
\address{School of General Education, Chubu University, 1200 Matsumoto-cho, Kasugai-shi, Aichi 487-8501, JAPAN}
\email{moriwaki@fsc.chubu.ac.jp}
\title{Arithmetic dynamics and Generalized Fermat's conjecture}
\date{\DateTime}
\thanks{I would like to thank Prof. Kawaguchi for his helpful comments. I also appreciate the anonymous useful  suggestions.}

\begin{document}
\maketitle
\begin{abstract}
In this short paper, we would like to propose generalized Fermat's conjecture over an arithmetic function field in the framework of arithmetic dynamics. We also give several evidences for the conjecture.
\end{abstract}

\section{Introduction}
Let $K$ be an arithmetic function field (i.e., $K$ is finitely generated over $\QQ$) and $S = ((\Omega, \mathcal{A}, \nu), \{ |\ndot|_{\omega} \}_{\omega \in \Omega})$ be a proper adelic structure of $K$
with Northcott's property and $\nu(\mathcal{A}) \not\subseteq \{ 0, 1 \}$
(for the definition of ``adelic structure'' and its related terminology, see Section~\ref{sec:adelic:structure}). 
Let $X$ be a geometrically integral projective scheme over $K$ 
and $L$ be an ample line bundle on $X$.
Fix a positive integer $N_0$. Let $F = \{ f_N \}_{N=N_0}^{\infty}$ be a sequence of endomorphisms of $X$ with the following properties:
\begin{enumerate}[label=\rm (\arabic*)]
\item $f_N^*(L) \simeq L^{\otimes d_N}$ for some integer $d_N \geqslant 2$. 
\item $\lim_{N\to\infty} \deg(f_N) = \infty$. 
\item $f_N \circ f_{N'} = f_{N'} \circ f_{N}$ for all $N, N' \geqslant N_0$. 
\end{enumerate}
The sequence $F$ is called a \emph{system of endomorphisms polarized by $L$}.
We say $F$ is \emph{multiplicative} (resp. \emph{additive}) if $f_N \circ f_{N'} = f_{N'} \circ f_{N} = f_{NN'}$ (resp. $f_N \circ f_{N'} = f_{N'} \circ f_{N} = f_{N+N'}$) for all $N, N' \geqslant N_0$.
For example, the cases of 
Example~\ref{example:projective:space},
Example~\ref{example:abelian:variety}, Example~\ref{example:latte:map} and Example~\ref{example:Chebyshev:map} are multiplicative,  and the case of Example~\ref{example:dynamical:system} is additive.
We can assign the height function $h_F$ on $X(K)$ to $F$ such that 
\[h_F(x) \geqslant 0\ \text{and}\  h_F(f_N(x)) = d_Nh_F(x)\quad (\forall x \in X(K))\]
(c.f., Proposition~\ref{prop:sys:endo:property}). 
Let $Y$ be an equidimensional  subscheme of $X$ and $Y_N := f_N^{-1}(Y)$  ($N \geqslant N_0$). 
We say that $Y_N$ has \emph{Fermat's property over $K$} if \[Y_N(K) \subseteq \{ x \in X(K) \mid h_F(x) = 0 \}\]
(see Remark~\ref{rem:why:Fermat:prop} for the reason why we use the terminology ``Fermat's property'').
In this paper, we would like to propose the following conjecture and give its evidences.

\begin{conjecture}[Generalized Fermat's conjecture]\label{conj:gen:Fermat:conj}
If there is $N_1 \geqslant N_0$ such that $Y_N(K)$ is finite for any integer $N \geqslant N_1$, 
then does there exist $N_2 \geqslant N_0$ such that
$Y_N$ has Fermat's property over $K$ for any $N \geqslant N_2$?
For a multi-indexed version, see Conjecture~\ref{mult:conj:gen:Fermat:conj}.
\end{conjecture}

\begin{theorem}\label{thm:main}
\begin{enumerate}[label=\rm (\arabic*)]
\item If $Y(K)$ is finite, then there exists $N_2$ such that
$Y_N$ has Fermat's property for any $N \geqslant N_2$.

\item
If $F$ is additive and there is $N_1 \geqslant N_0$ such that $Y_{N_1}(K)$ is finite, 
then there exists $N_2 \geqslant N_0$ such that
$Y_N$ has Fermat's property for any $N \geqslant N_2$. Namely,
Conjecture~\ref{conj:gen:Fermat:conj} holds for an additive system of endomorphisms polarized by $L$.

\item
If $F$ is multiplicative and there is a positive integer $p_0$ such that $Y_p(K)$ is finite for any prime $p$ with $p \geqslant p_0$,
then
\[
\lim_{m\to\infty} \frac{ \# \{ N_0 \leqslant N \leqslant m \mid \text{$F_N$ has Fermat's property over $K$}\} }{m} = 1,
\]
that is, Fermat's property holds with probability one.
\end{enumerate}
\end{theorem}

\section{Adelic structure of field}\label{sec:adelic:structure}
In this section, we quickly review several notions related to an adelic structure of a field.
For details, see the manuscripts \cite{MR4292529, CMIntersection, Chen_Moriwaki24} due to H. Chen and A. Moriwaki.

An \emph{adelic structure} of a field $K$ consists of data
$S = ((\Omega, \mathcal A, \nu), \{ |\ndot|_\omega\}_{\omega \in \Omega})$ satisfying the following properties:
\begin{enumerate}[label=\rm (\arabic*)]
\item $(\Omega, \mathcal A, \nu)$ is a measure space, that is,
$\mathcal A$ is a $\sigma$-algebra of $\Omega$ and $\nu$ is a measure on the measurable space $(\Omega, \mathcal A)$.

\item 
$\{ |\ndot|_\omega\}_{\omega \in \Omega}$ is a collection of absolute values of $K$ indexed by $\Omega$.

\item For $a \in K^{\times}$, the function $(\omega \in \Omega) \mapsto \log |a|_{\omega}$ is $\nu$-integrable.
\end{enumerate}
For $\omega \in \Omega$, the completion of $K$ with respect to $|\ndot|_{\omega}$ is denoted by $K_{\omega}$.
The adelic structure $S$ is said to be \emph{proper} if
\begin{equation}\label{eqn:product:formula}
\int_{\Omega} \log |a|_{\omega} \nu(\mathrm{d}\omega) = 0
\end{equation}
holds for all $a \in K^{\times}$. The equation \eqref{eqn:product:formula} is
called the \emph{product formula}. From now on, we assume that $S$ is proper.

Let $V$ be a finite-dimensional vector space over $K$. For each $\omega \in \Omega$, let
$\|\ndot\|_{\omega}$ be a norm of $V_\omega := V \otimes_K K_\omega$ over $(K_{\omega}, |\ndot|_\omega)$.
If $|\ndot|_\omega$ is non-archimedean, we always assume that $\|\ndot\|_\omega$ is ultrametic, that is,
$\| x + y \|_\omega \leqslant \max \{ \|x \|_\omega, \|y \|_\omega \}$ for all $x, y \in E_{\omega}$.

The collection $\|\ndot\| = \{ \|\ndot\|_{\omega} \}_{\omega \in \Omega}$ is said to be \emph{dominated} (resp.  \emph{measurable})
if there exist a $\nu$-integrable function $C : \Omega \to \RR_{\geqslant 0}$ and a basis $\{ e_1, \ldots, e_r \}$ of $V$ such that,
for any $\omega \in \Omega$ and $(\lambda_1, \ldots, \lambda_r) \in K_\omega^r \setminus \{(0, \ldots, 0) \}$,
\[
\left| \log \| \lambda_1 e_1 + \cdots + \lambda_r e_r\|_\omega - \log \max_{i = 1, \ldots, r} |\lambda_i|_\omega \right| \leqslant C(\omega)
\]
(resp. for each $s \in V$, the function $\Omega \to \RR$ given by $\omega \mapsto \| s \|_\omega$ is
$\nu$-measurable). Moreover, if it is dominated and measurable, then it is said to be \emph{adelic}.

If $\dim_K V = 1$, the above conditions are equivalent to
the $\nu$-integrability of the function $\Omega \to \RR$ given by $\omega \mapsto \log \| s \|_\omega$ for $s \in V \setminus \{ 0 \}$.
Moreover, note that if $(V, \|\ndot\|)$ is adelic, then
$(\det V, \|\ndot\|_{\det})$ is adelic.

\if01
Let $V$ be a one-dimensional vector space over $K$. For each $\omega \in \Omega$, let
$\|\ndot\|_{\omega}$ be a norm of $V_\omega := V \otimes_K K_\omega$ over $K_{\omega}$.
The collection $\|\ndot\| = \{ \|\ndot\|_{\omega} \}_{\omega \in \Omega}$ is said to be \emph{adelic} if,
for each $s \in V \setminus \{ 0 \}$, the function $\Omega \to \RR$ given by $\omega \mapsto \log \| s \|_\omega$ is
$\nu$-integrable. If $\|\ndot\|$ is adelic, then
\[
\int_{\Omega} \log \|s \|_\omega \nu(d\omega)
\]
does not depend on the choice of $s \in V \setminus \{ 0 \}$ because of the product formula, so
it is called the \emph{Arakelov degree} of $(V, \|\ndot\|)$, and is denoted by $\adeg(V,\|\ndot\|)$.
\fi

Let $F$ be a finitely generated field over $K$ such that
the transcendental degree of $F$ over $K$ is $1$. Let $|\ndot|$ be an absolute value of $F$ such that the restriction of $|\ndot|$ to $K$ is trivial.
Then 
the valuation ring \[(\{ x \in F \mid |x| \leqslant 1 \}, \{x \in F \mid |x| < 1 \})\] is discrete.
If $\ord$ is the order function on the above discrete valuation ring,
then  $|\ndot| = \exp(-q \ord(\ndot))$ for some $q \in \RR_{\geqslant 0}$.
Note that $q$ is called the \emph{exponent} of $|\ndot|$.

Let $X$ be a geometric integral projective scheme over $K$, $X_\omega := X \times_{\Spec K} \Spec K_{\omega}$ for $\omega \in \Omega$, and 
$X_{\omega}^{\mathrm{an}}$ be the analytification in the sense of Berkovich \cite{MR1070709}. Let $L$ be an invertible $\OO_X$-module and $L_{\omega}^{\mathrm{an}}$ be the pull-back of $L$ via
the natural map $X_{\omega}^{\mathrm{an}} \to X$. For each $\omega \in \Omega$, let $\varphi_\omega$ be a continuous metric of $L_\omega^{\mathrm{an}}$.

The collection $\varphi = \{ \varphi_\omega \}_{\omega \in \Omega}$ is said to be \emph{dominated}
if there exist very ample invertible $\mathcal O_X$-modules $L_1$ and $L_2$, together with  dominated  norm families $\xi_1$ and $\xi_2$ on $H^0(X,L_1)$ and $H^0(X,L_2)$, respectively, which satisfy the following conditions.
\begin{enumerate}
\item  One has $L\cong L_2\otimes L_1^\vee$.
\item 
If we denote by $\varphi_1$ and $\varphi_2$ the quotient metric families on $L_1$ and $L_2$ of $\xi_1$ and $\xi_2$, respectively,  then the function 
\[(\omega\in\Omega)\longmapsto 
\sup_{x \in X_{\omega}^{\mathrm{an}}} \left|  \log \frac{|\ndot|_{\varphi_{2,\omega}\otimes\varphi_{1,\omega}^\vee}(x)}{|\ndot|_{\varphi_{\omega}}(x)} \right| 
\]
is bounded from above by a $\nu$-integrable function.
\end{enumerate}

Assuming $K$ has the trivial absolute value, we set
\[
X^{\mathrm{an}}_{1,\QQ} = \left\{ x \in X^{\mathrm{an}} \ \left|\ \begin{array}{l} \text{the Zariski closure of $j(x)$ has dimension $1$} \\ \text{and the exponent of $|\ndot|_x$ is rational} \end{array}\right\}\right.,
\]
where $j$ is the canonical map $X^{\mathrm{an}} \to X$.
Let $\Omega_0 = \{ \omega \in \Omega \mid \text{$|\ndot|_\omega$ is trivial}\}$, and $\mathcal{A}_0 = \{ A \cap \Omega_0 \}_{A \in \mathcal{A}}$ be the restriction of
$\mathcal{A}$ to $\Omega_0$.
The family $\varphi = \{ \varphi_\omega \}_{\omega \in \Omega}$  of metrics is said to be \emph{measurable} if the following conditions are satisfied:
\begin{enumerate}
\item for any closed point $x$ of $X$, $x^*(\varphi)$ is measurable.
\item for any $x \in X^{\mathrm{an}}_{1,\QQ}$ and any $\ell \in L \otimes \hat{\kappa}(x)$,
the map $\omega \in \Omega_0 \mapsto |\ell|_{\varphi_\omega}(x)$ is measurable with respect to $\mathcal{A}_0$, where
$\hat{\kappa}(x)$ is the completion of the residue field $\kappa(j(x))$ with respect to $|\ndot|_x$.
\end{enumerate}

Moreover, $\varphi$ is \emph{adelic} if it is dominated and measurable.
Note that, if $\varphi$ is adelic, then,
for $x \in X(K)$ (i.e., a morphism $x : \Spec(K) \to X$ over $K$), the pullback $x^*(L, \varphi)$ yields a one-dimensional vector space over $K$ with
an adelic norm family. Thus  the \emph{height} of $x$ with respect to $(L, \varphi)$ is defined by
$\adeg(x^*(L, \varphi))$, which is denoted by $h_{(L,\varphi)}(x)$.

We say $S$ has \emph{Northcott's property} if, for any $C \in \RR$,
the set 
\[
\left\{ a \in K \ \left|\ \int_\Omega \log \max \{ |a|_{\omega}, 1 \}  \nu(d\omega) \leqslant C \right\}\right.
\]
is finite.
Note that, by \cite[Proposition~6.2.3]{MR4292529}, if $S$ has Northcott's property and $L$ is ample, then
the set $\{ x \in X(K) \mid h_{(L,\varphi)}(x) \leqslant C \}$ is finite for all $C \in \RR$.
Moreover, if $K$ is an arithmetic function field, then, by \cite[Theorem~2.7.18]{CMIntersection}, $K$ has a proper adelic structure with Northcott's property and
$\nu(\mathcal A) \not\subseteq \{ 0, 1 \}$.
Finally note that once we have a proper adelic structure with Northcott's property, the theory of height functions works similarly to the case of number fields.

\section{Endomorphism polarized by an ample line bundle}
Let $f : X \to X$ be an endomorphism over $K$.
Let $L$ be an ample line bundle on $X$.
We say $f$ is \emph{polarized by $L$} if there is an isomorphism $\alpha : f^*(L) \simeq L^{\otimes d(f)}$ for some positive integer $d(f)$.

\begin{proposition}\label{prop:fund:prop:dynamics}
\begin{enumerate}[label=\rm (\arabic*)]
\item $d(f)^{\dim X} = \deg(f)$.
\end{enumerate}
From now on, we assume that $d(f)$ is greater than or equal to $2$.
\begin{enumerate}[label=\rm (\arabic*)]
\setcounter{enumi}{1}
\item There exists the unique metric family $\varphi_{f,\alpha} = \{ \varphi_{f,\alpha, \omega} \}_{\omega \in \Omega}$ of $L$, called the \emph{canonical global compactification} of $L$,
such that, for each $\omega \in \Omega$, the isomorphism $\alpha_\omega : f_\omega^*(L_\omega) \simeq L_\omega^{\otimes d(f)}$ extends to the isometry
$f_\omega^*(L_{\omega}, \varphi_{f,\alpha, \omega}) \simeq (L_{\omega}, \varphi_{f,\alpha, \omega})^{\otimes d(f)}$.
Moreover, $(L, \varphi_{f,\alpha})$ is a nef  adelic line bundle.
In particular, $h_{f}(x) \geqslant 0$ for any $x \in X(K)$. 
\item The height function $h_{\varphi_{f,\alpha}}$ with respect to $\varphi_{f, \alpha}$ does not depend on the choice of the isomorphism $\alpha$, so that $h_{\varphi_{f,\alpha}}$ is denoted by $h_{f}$.
\item For any $x \in X(K)$, $h_f(f(x)) = d(f) h_f(x)$.
\item Let $g$ be another endomorphism of $X$ such that $g$ is polarized by $L$ and $d(g) \geqslant 2$.
If $f \circ g = g \circ f$, then $h_f = h_g$. 
Moreover, $f \circ g$ is polarized by $L$, and $d(f \circ g) = d(f) d(g)$.
\item If we set
$a(h_f) := \inf \{ h_{f}(x) \mid \text{$x \in X(K)$ and $h_{f}(x) > 0$} \}$,
then $a(h_f) > 0$.
\end{enumerate}
\end{proposition}

\begin{proof}
(1) Indeed,
\[
d(f)^{\dim X} (L^{\dim X}) = \Big( (L^{\otimes d(f)})^{\dim X} \Big) =( f^*(L)^{\dim X}) = \deg(f) (L^{\dim X}),
\]
as required.

(2) For the existence of $\varphi_{f,\alpha}$, we refer to \cite[Proposition~2.5.11 (4)]{MR4292529}. The second assertion follows from \cite[Proposition~9.3.7]{Chen_Moriwaki24}.

(3) Let $\beta$ be another isomorphism $\beta : f^*(L) \simeq L^{\otimes d(f)}$. Then we can find $c \in K^{\times}$ such that $\beta = c \alpha$.
Thus, by \cite[Proposition~9.3.3]{Chen_Moriwaki24}, for each $\omega \in \Omega$, \[|\ndot|_{\varphi_{f, \beta,\omega}} = |c|_{\omega}^{-1/(d(f)-1)}|\ndot|_{\varphi_{f,\alpha,\omega}}.\]
Therefore we obtain (3).

(4) Indeed,
\begin{align*}
h_f(f(x)) & = h_{(L, \varphi_{f,\alpha})}(f(x)) = h_{f^*(L, \varphi_{f,\alpha})}(x) = h_{(L, \varphi_{f,\alpha})^{\otimes d(f)}}(x) \\
& = d(f)h_{(L, \varphi_{f,\alpha})}(x) = d(f)h_f(x).
\end{align*}

(5) Fix isomorphisms $\alpha : f^*(L) \simeq L^{\otimes d(f)}$ and $\beta : g^*(L) \simeq L^{\otimes d(g)}$.
Let us consider the following homomorphisms:
\[
\begin{cases}
\begin{CD} g^*(f^*(L)) @>{\sim}>{g^*(\alpha)} > g^*(L^{\otimes d(f)}) @>{\sim}>{\beta^{\otimes d(f)}}> L^{\otimes d(f)d(g)} \end{CD}, \\[2ex]
\begin{CD} g^*(f^*(L)) = f^*(g^*(L)) @>{\sim}>{f^*(\beta)}> f^*(L^{\otimes d(g)})@>{\sim}>{\alpha^{\otimes d(g)}}>  L^{\otimes d(f)d(g)} \end{CD}.
\end{cases}
\]
Thus we can find $r \in K^{\times}$ such that $\beta^{\otimes d(f)} \circ g^*(\alpha) = r \cdot \alpha^{\otimes d(g)} \circ f^*(\beta)$.
Then, by \cite[Proposition~9.3.4]{Chen_Moriwaki24},
one has $|\ndot|_{\varphi_{g, \beta, \omega}} = |r|_\omega^{-1/(d(f)-1)(d(g)-1)}|\ndot|_{\varphi_{f, \alpha, \omega}}$.
Thus the first assertion follows. 
Moreover, by using the above isomorphisms, the second assertion is obtained.

(6) 
Note that, for any $C \in \RR$,
the set $\{ x \in X(K) \mid h_{f}(x) \leqslant C \}$ is finite.
Thus (6) follows.
\end{proof}

The following lemma is a key of this paper.

\begin{lemma}\label{lem:height:zero}
Let $S$ be a finite subset of $X(K)$. Let $f : X \to X$  and $g : X \to X$ be endomorphisms of $X$ such that $f$ and $g$ are polarized by $L$, $d(f), d(g) \in \ZZ_{\geqslant 2}$, and $f \circ g = g \circ f$.
If $d(g) a(h_f) > \max \{ h_{f}(s) \mid s \in S \}$ and $g(x) \in S$ for $x \in X(K)$,
then $h_{f}(x) = 0$.
\end{lemma}

\begin{proof}
If we assume that $h_{f}(x) > 0$, then, by (4) and (5) in Proposition~\ref{prop:fund:prop:dynamics},
\[
\max \{ h_{f}(x) \mid x \in S \} \geqslant h_{f}(g(x)) = h_g(g(x)) = d(g) h_{g}(x) = d(g) h_{f}(x) \geqslant d(g) a(h_f),
\]
which is a contradiction.
\end{proof}

\section{Generalized Fermat's conjecture}

Let $F = \{ f_N \}_{N=N_0}^{\infty}$ be a system of endomorphisms polarized by $L$.

\begin{proposition}\label{prop:sys:endo:property}
\begin{enumerate}[label=\rm (\arabic*)]
\item The height function $h_{f_N}$ ($N \geqslant N_0)$ does not depend on $N$, so they are denoted by $h_F$. 
\item For $N \geqslant N_0$, $h_F(f_N(x)) = d_N h_F(x)$.
\end{enumerate}
\end{proposition}

\begin{proof}
(1) follows from (5) in Proposition~\ref{prop:fund:prop:dynamics}.

(2) 
By (4) in Proposition~\ref{prop:fund:prop:dynamics},
\[
h_F(f_N(x)) = h_{f_N}(f_N(x)) = d_N h_{f_N}(x) = d_N h_{F}(x),
\]
as required.
\end{proof}

The following are examples of systems of endomorphisms polarized by $L$.

\begin{example}\label{example:projective:space}
Let $X = \PP^n_K$, $L = \OO_{\PP^n}(1)$ and $f_N(x_0 : \cdots : x_n) = (x_0^N : \cdots : x_n^N)$ ($N \geqslant 2$).
Then $f_N^*(L) \simeq L^{\otimes N}$ and $f_N \circ f_{N'} = f_{N} \circ f_{N'} = f_{NN'}$. Note that $d_N = N$ and $\{ f_N \}_{N=2}^{\infty}$ is
multiplicative.
In this case, $h_F(x_0 : \cdots : x_n) = 0$ if and only if
there is $t \in K^{\times}$ such that $tx_0, \ldots, tx_n \in \{ 0 \} \cup \mu(K)$, where $\mu(K)$ is the set of all roots of the unity in $K$.
\end{example}

\begin{example}\label{example:abelian:variety}
Let $X$ be an abelian variety over $K$, $L$ be an even ample line bundle on $X$, and $f_N(x) = Nx$ ($N \geqslant 2$).
Then $f_N^*(L) \simeq L^{\otimes N^2}$ and $f_N \circ f_{N'} = f_{N} \circ f_{N'} = f_{NN'}$. Note that $d_N = N^2$ and $\{ f_N \}_{N=2}^{\infty}$ is
multiplicative.
In this case, $h_F(x) = 0$ if and only if
$x$ is a torsion point.
\end{example}

\begin{example}\label{example:latte:map}
Here we consider a family of Latt\`es maps. Let $E$ be an elliptic curve over $K$ and $\pi : E \to \PP^1$ be the morphism such that
$\deg(\pi) = 2$ and $\pi(P) = \pi(-P)$, that is, $\PP^1 = E/\{ \pm 1 \}$.
For an integer $N \geqslant 2$, let $f_N : \PP^1 \to \PP^1$ be the morphism such that the following diagram is commutative:
\[\begin{CD}
E @>[N]>> E \\
@V{\pi}VV @VV{\pi}V \\
\PP^1 @>{f_N}>> \PP^1
\end{CD}\]
Then $f_N^*(\OO_{\PP^1}(1)) = \OO_{\PP^1}(N^2)$ and $f_N \circ f_{N'} = f_{N'} \circ f_{N}  = f_{NN'}$.
Note that $d_N = N^2$ and $\{ f_N \}_{N=2}^{\infty}$ is
multiplicative.
\end{example}

\begin{example}\label{example:Chebyshev:map}
Let $T_N(z)$ be the Chebyshev polynomial of degee $N$, which is given by the recursive relation
$T_{N}(z) = z T_{N-1}(z) - T_{N-2}(z)$ ($T_0 = 2$, $T_1 = z$ and $T_2= z^2 - 2$). It can be characterized by $T_N(z+z^{-1}) = z^{N} + z^{-N}$, so we can see that $T_{N}(T_{N'}(z)) = T_{N'}(T_{N}(z)) = T_{NN'}(z)$.
Thus $\{ T_{N} \}_{N=2}^{\infty}$ yields a multiplicative system of endomorphisms.
Note that $(1/2)T_N(2z)$ is the classical Chebyshev polynomial of degree $N$.
For details, see \cite[Section~6.2]{silverman2007arithmetic}.
\end{example}

\begin{example}\label{example:dynamical:system}
Let $f$ be an endomorphism of $X$ such that $f$ is polarized by $L$ and $d(f) \geqslant 2$. We set $f_N = f^{N}$ ($N \geqslant 1$).
Then $(f_N)^* (L) \simeq L^{\otimes d(f)^N}$ and $f_N \circ f_{N'} = f_{N} \circ f_{N'} = f_{N+N'}$. Note that $d_N = d(f)^N$ and $\{ f_N \}_{N=1}^{\infty}$ is additive.
In this case, $h_F(x) = 0$ if and only if
$x$ is preperiodic with respect to $f$.
It seems that Theorem~\ref{thm:main} is slightly related to the problem posed in \cite{BMSCancellation}.
We assume that $f(Y) \subseteq Y$. 
Then we have the tower
\[
Y(K) \subseteq Y_{1}(K) \subseteq \cdots \subseteq Y_{N}(K) \subseteq \cdots
\]
If $Y(K)$ is finite, 
by (1) in Theorem~\ref{thm:main}, 
there is $N_2$ such that \[Y_N(K) \subseteq \{ x \in X(K) \mid h_F(x) =0 \}\] for all $N \geqslant N_2$, 
so the tower is stabilized because
$\{ x \in X(K) \mid h_F(x) = 0\}$ is finite.
\end{example}

\begin{remark}
The above Example~\ref{example:projective:space}, Example~\ref{example:abelian:variety}, Example~\ref{example:latte:map} and
Example~\ref{example:Chebyshev:map} are exceptional from the dynamics point of view, but Example~\ref{example:dynamical:system} is the general iterative
setting. In this sense, the actual analogue of Fermat's conjecture is very restricted in the dynamical systems.
\end{remark}

\begin{remark}\label{rem:why:Fermat:prop}
In Example~\ref{example:projective:space}, we assume that 
$K = \QQ$, $n = 2$ and $Y$ is defined by $X_0 + X_1 - X_2 = 0$. Note that
$Y_N$ is given by $X_0^N + X_1^N - X_2^N$. If $Y_N$ has Fermat's property, then, for  $(x_0 : x_1 : x_n) \in Y_N(K)$, we can find $t \in \QQ^{\times}$ such that $tx_0, tx_1, tx_2 \in \{ 0, 1, -1 \}$.
Thus we can easily see that $x_0 x_1 x_2 = 0$, and hence Fermat's conjecture for $N$ follows from Conjecture~\ref{conj:gen:Fermat:conj}.
\end{remark}

\begin{remark}
If $\dim Y = 1$, then $\dim Y_N = 1$ because $f_N$ is finite.
Thus if there is $N_0$ such that every irreducible component of $Y_N$ has the genus $\geqslant 2$ for any integer $N \geqslant N_0$,
then, by Faltings theorem (c.f. \cite{MR1307396}), $Y_N(K)$ is finite for $N \geqslant N_0$. 
\end{remark}

\begin{remark}
If $\{ f_N \}_{N=N_0}^{\infty}$ is a multiplicative system of endomorphisms polarized by $L$,
then, for a fixed integer $m \geqslant 2$, $\{ g_N \}_{N=N_1}^{\infty}$ given by $g_N = f_{m^N}$ is additive,
where $N_1$ is an integer with $m^{N_1} \geqslant N_0$.
\end{remark}

\begin{remark}
Theorem~\ref{thm:main} was previously discussed in the following cases:
\begin{enumerate}[label=\rm (\arabic*)]
\item $K = \QQ$, $X = \PP^2$, $f_N(x_0 : x_1 : x_2) = (x_0^N : x_1^N : x_2^N)$ and $Y$ is given by $X_0 + X_1 - X_2$
(M. Filaseta \cite{MR734070}, A. Granville \cite{MR777765}, D. R. Heath-Brown \cite{MR766439}).

\item  $K$ is a number field, $X = \PP^2$, $f_N(x_0 : x_1 : x_2) = (x_0^N : x_1^N : x_2^N)$ and $Y$ is given by $X_0 + X_1 - X_2$ (H. Ikoma, S. Kawaguchi and  A. Moriwaki \cite{IKMFaltings}).

\item $K$ is an arithmetic function field, $X = \PP^2$, $f_N(x_0 : x_1 : x_2) = (x_0^N : x_1^N : x_2^N)$ and $Y$ is a strongly non-degenerated curve, that is, $Y_N$
is smooth over $K$ for all $N$ (H. Chen and A. Moriwaki \cite{CMIntersection}).
\end{enumerate}
\end{remark}

\begin{proof}[Proof of Theorem~\ref{thm:main}]
Let $S$ be a finite subset of $X(K)$.
As $\lim_{m\to\infty} d_m = \infty$ by (1) of Proposition~\ref{prop:fund:prop:dynamics}, 
there is $m(S)$ such that
\[d_m a(h_F) > \max \{ h_{F}(y) \mid y \in S \}\] for all $m \in \ZZ_{\geqslant m(S)}$.

(1) We set $S = Y(K)$, $f = f_{N_0}$ and $g = f_N$ for $N \in \ZZ_{\geqslant m(S)}$.
As $g(Y_N(K)) \subseteq S$ and $g \circ f = f \circ g$,
by Lemma~\ref{lem:height:zero}, $Y_N(K) \subseteq \{ x \in X(K) \mid h_F(x) = 0 \}$.

(2) Note that $f^{-1}_N(Y_{N_1}) = Y_{N+N_1}$. Thus the assertion follows from (1).

(3) For any prime number $p \geqslant p_0$, we set $S = Y_p(K)$, $f = f_{N_0}$ and $g = f_m$ for $m
\in \ZZ_{\geqslant m(S)}$.
If $x \in Y_{mp}(K)$, 
then one has $g(x) \in Y_p(K)$ because $f_p(g(x)) = f_{mp}(x) \in Y(K)$.
Moreover, $g \circ f = f \circ g$.
Thus, by Lemma~\ref{lem:height:zero}, $Y_{mp}(K) \subseteq \{ x \in X(K) \mid h_F(x) = 0 \}$.
Therefore, by \cite[Lemma~5.16]{IKMFaltings} or Lemma~\ref{lemma:proba:1:mult:index}, we have the assertion.
\end{proof}

\section{Product case (multi-indexed version)}
Let $X_1, \ldots, X_n$ be geometrically integral projective schemes over $K$ 
and $L_1, \ldots, L_n$ be ample line bundles on $X_1, \ldots, X_n$, respectively.
Let $f_1, \ldots, f_n$ be endomorphisms of $X_1, \ldots, X_n$, respectively.
We assume that, for each $i \in \{1, \ldots, n\}$,  $f_i^*(L_i) \simeq L_i^{\otimes d_i}$ for some integer $d_i \geqslant 2$.
Let $\varphi_i$ be the global compatification of $f_i^*(L_i) \simeq L_i^{d_i}$ for each $i$.
Let $X = X_1 \times \cdots \times X_n$, $L = p_1^*(L_1) \otimes \cdots \otimes p_n^*(L_n)$ and $f = f_1 \times \cdots \times f_n$, that is,
$f(x_1, \ldots, x_n) = (f_1(x_1), \ldots, f_n(x_n))$ for $(x_1, \ldots, x_n) \in X$. Moreover, let $p_i : X \to X_i$ be the projection to the $i$-th factor.
Note that $f^*(L) \simeq p_1^*(L_1)^{\otimes d_1} \otimes \cdots \otimes p_n^*(L_n)^{\otimes d_n}$, and
$\varphi = p_1^*(\varphi_1) \otimes \cdots \otimes p_n^*(\varphi_n)$ yields the global compatification of
$f^*(L) \simeq p_1^*(L_1)^{\otimes d_1} \otimes \cdots \otimes p_n^*(L_n)^{\otimes d_n}$, that is,
we have the isometry \[f_{\omega}^*(L_{\omega}, \varphi_\omega) \simeq p_{1,\omega}^*(L_{1,\omega}, \varphi_{1, \omega})^{\otimes d_1} \otimes \cdots \otimes p_{n,\omega}^*(L_{n,\omega}, \varphi_{n, \omega})^{\otimes d_n}\]
for each $\omega \in \Omega$.

\begin{proposition}
For $x_1 \in X_1(K), \ldots, x_n \in X_n(K)$,
\[\begin{cases}
h_{f} (x_1, \ldots, x_n) = h_{f_1}(x_1) + \cdots + h_{f_n}(x_n),\\
h_{f} (f(x_1, \ldots, x_n)) = d_1h_{f_1}(x_1) + \cdots + d_nh_{f_n}(x_n).
\end{cases}
\]
\end{proposition}

\begin{proof}
{\allowdisplaybreaks
Indeed,
\begin{align*}
h_{f} (x_1, \ldots, x_n) & = h_{(p_1^*(L_1) \otimes \cdots \otimes p_n^*(L_n), p_1^*(\varphi_1) \otimes \cdots \otimes p_n^*(\varphi_n))} (x_1, \ldots, x_n) \\
& = h_{(p_1^*(L_1), p_1^*(\varphi_1))}  (x_1, \ldots, x_n) + \cdots + h_{(p_n^*(L_n), p_n^*(\varphi_n))}  (x_1, \ldots, x_n) \\
& = h_{f_1}(x_1) + \cdots + h_{f_n}(x_n).
\end{align*}
Moreover,
\begin{align*}
h_{f} (f(x_1, \ldots, x_n)) & = h_{f} (f_1(x_1), \ldots, f_n(x_n)) = h_{f_1}(f_1(x_1)) + \cdots + h_{f_n}(f_n(x_n)) \\
& = d_1h_{f_1}(x_1) + \cdots + d_nh_{f_n}(x_n),
\end{align*}
as required.}
\end{proof}

\begin{lemma}\label{lem:height:zero:mult:index}
Let $S$ be a finite subset of $X(K)$. For each $i \in \{1, \ldots, n\}$, let $g_i : X_i \to X_i$ be an endomorphism of $X_i$ such that $g_i$ are polarized by $L_i$, $d(g_i) \in \ZZ_{\geqslant 2}$, and $f_i \circ g_i = g_i \circ f_i$. We set $g = g_1 \times \cdots \times g_n$.
If $\min\{  d(g_i) a(h_{f_i}) \mid i \in \{ 1, \ldots, n \} \} > \max \{ h_{f}(s) \mid s \in S \}$ and $g(x) \in S$ for $x \in X(K)$,
then $h_{f}(x) = 0$.
\end{lemma}

\begin{proof}
We set $x = (x_1, \ldots, x_n)$ and $I = \{ i \in \{ 1, \ldots, n \} \mid h_{f_{i}}(x_i) > 0 \}$. If we assume $I \not= \emptyset$,
then 
{\allowdisplaybreaks
\begin{align*}
\max \{ h_{f}(x) \mid x \in S \} & \geqslant h_{f}(g(x)) = \sum_{i=1}^n h_{f_i}(g_i(x_i)) = \sum_{i=1}^n h_{g_i}(g_i(x_i)) \\
& = \sum_{i=1}^n d(g_i) h_{g_i}(x_i) = \sum_{i=1}^n d(g_i) h_{f_i}(x_i) \\
& = \sum_{i \in I} d(g_i) h_{f_i}(x_i) \geqslant \sum_{i \in I} d(g_i) a(h_{f_i})  \\
& \geqslant \min\{  d(g_i) a(h_{f_i}) \mid i \in \{ 1, \ldots, n \} \},
\end{align*}}
which is a contradiction.
\end{proof}

For $I = (i_1, \ldots, i_n) \in (\ZZ_{\geqslant 1})^n$ and $1 \leqslant k \leqslant n$, we denote $i_k$ by $I(k)$.
We define $\min I$ to be $\min \{ I(k) \mid k \in \{ 1, \ldots, n \} \}$.
For $I, J \in (\ZZ_{\geqslant 1})^n$, we define $I \leqslant J$, $I + J$ and $I \cdot J$ to be
\[
\begin{cases}
I \leqslant J \ \Longleftrightarrow\ I(1) \leqslant J(1), \ldots, I(n) \leqslant J(n), \\
I + J = (I(1) + J(1), \ldots, I(n) + J(n)), \\
I \cdot J =  (I(1) \cdot J(1), \ldots, I(n) \cdot  J(n)).
\end{cases}
\]
Moreover, we say $I$ is \emph{prime} if $I(k)$ is a prime number for every $k \in \{ 1, \ldots, n \}$.

Let $F_1 = \{ f_{1, N} \}_{N=N_1}^{\infty}, \ldots, F_n = \{ f_{n, N} \}_{N=N_n}^{\infty}$
be systems of endomorphisms on $X_1, \ldots, X_n$ polarized by $L_1, \ldots, L_N$, respectively,
that is,
\[f^*_{k, N}(L_k) \simeq L_k^{\otimes d_{k, N}}\ (d_{k, N} \in \ZZ_{\geqslant 2}),\ 
\lim_{N\to\infty} \deg(f_{k, N}) = \infty\ \text{and}\  f_{k, N} \circ k_{k, N'} = f_{k, N'} \circ f_{k,N}.
\]
We set $I_0 = (N_1, \ldots, N_n)$. For $I \geqslant I_0$, 
the endomorphism  $f_{1, I(1)} \times \cdots \times f_{n, I(n)}$ on $X$ is denoted by $f_I$, and
the collection $\{ f_{I} \}_{I \in (\ZZ_{\geqslant 1})^n_{\geqslant N_0}}$ is denoted by $F$. 
Note that the following properties are easily verified:
\begin{enumerate}[label=\rm (\arabic*)]
\item For $I, I' \in (\ZZ_{\geqslant 1})^n_{\geqslant N_0}$, $f_I \circ f_{I'} = f_{I'} \circ f_{I}$.
\item $f_I^*(p_1^*(L_1) \otimes \cdots \otimes p_n^*(L_n)) = p_1^*(L_1)^{\otimes d_{1, I(1)} } \otimes \cdots \otimes p_n^*(L_n)^{\otimes d_{n, I(n)} }$.
\end{enumerate}
We say $F$ is \emph{multiplicative} (resp. \emph{additive}) if $f_I \circ f_{I'} = f_{I'} \circ f_{I} = f_{I\cdot I'}$ (resp. $f_I \circ f_{I'} = f_{I'} \circ f_{I} = f_{I+I'}$) for all $I, I' \geqslant I_0$.
Moreover, note that the height function $h_{f_I}$ does not depend on the choice of $I$, so it  is denoted by $h_F$.
Let $Y$ be an equidimensional subscheme of $X$, and $Y_I = f^{-1}_I(Y)$.
We say $Y_I$ has \emph{Fermat's property over $K$} if \[Y_I(K) \subseteq \{ x \in X(K) \mid h_F(x) = 0 \}.\]

\begin{corollary}\label{cor:lem:height:zero:mult:index}
\begin{enumerate}[label=\rm (\arabic*)]
\item If $Y(K)$ is finite,
then there exists $I_2$ such that
$Y_I$ has Fermat's property for any $I \geqslant I_2$.

\item
If $F$ is multiplicative and $Y_I(K)$ is finite for $I \geqslant I_0$, 
then there is $M_0$ such that 
$Y_{I \cdot M}$ has Fermat's property for all $M \geqslant M_0$.
\end{enumerate}
\end{corollary}

\begin{proof}
Let $S$ be a finite subset of $X(K)$. Then there is $I(S) \geqslant I_0$ such that
\[
\min \{ d_{i, I(i)} a(h_{F_i}) \mid i \in \{ 1, \ldots, n \} \} > \max \{h_F(s) \mid s \in S \}
\]
for all $I \geqslant I(S)$.

(1) We set $S= Y(K)$, $f = f_{I_0}$ and $g = f_{I}$ for $I \geqslant I(S)$. Since $g(Y_I(K)) \subseteq Y(K)$ and $g \circ f = f \circ g$,
by Lemma~\ref{lem:height:zero:mult:index}, we have the assertion.

(2) We set $S = Y_I(K)$, $f = f_{I_0}$ and $g = f_{M}$ for $M \geqslant I(S)$. Note that $g \circ f = f \circ  g$. Moreover,
for $x \in Y_{I \cdot M}(K)$, since
\[
f_I(g(x)) = f_{I \cdot M}(x) \in Y(K),
\]
one has $g(x) \in Y_I(K)$.
Therefore, by Lemma~\ref{lem:height:zero:mult:index},
$h_F(x) = 0$, as required.
\end{proof}

\begin{conjecture}[Multi-indexed version of generalized Fermat's conjecture]\label{mult:conj:gen:Fermat:conj}
If there is $I_1 \geqslant I_0$ such that $Y_I(K)$ is finite for any $I \geqslant I_1$, 
then does there exist $I_2 \geqslant I_0$ such that
$Y_I$ has Fermat's property over $K$ for any $I \geqslant I_2$?
\end{conjecture}

One has the following as evidences of the above conjecture.

\begin{theorem}\label{thm:prob:1:mult}
\begin{enumerate}[label=\rm (\arabic*)]
\item
If $F$ is additive and there is $I_1 \geqslant I_0$ such that $Y_{I_1}(K)$ is finite, 
then there exists $I_2 \geqslant I_0$ such that
$Y_I$ has Fermat's property for any $I \geqslant I_2$. 

\item
If $F$ is multiplicative and there is $I_1 \geqslant I_0$ such that $Y_I(K)$ is finite for all $I \geqslant I_1$,
then
\[
\lim_{m_1\to\infty, \ldots, m_n\to\infty} \frac{ \#\left\{ I \in (\ZZ_{\geqslant 1})^n \ \left| \  \begin{array}{l} \text{$I_0 \leqslant I \leqslant (m_1, \ldots, m_n)$ and } \\ \text{$Y_I$ has Fermat's property over $K$} \end{array}\right.\right\}}{m_1 \cdots m_n} = 1.
\]
\end{enumerate}
\end{theorem}

\begin{proof}
If $F$ is additive, then $f^{-1}_I(Y_{I_1}) = Y_{I+I_1}$. Thus (1) follows from Corollary~\ref{cor:lem:height:zero:mult:index}.
(2)  is a consequence of Corollary~\ref{cor:lem:height:zero:mult:index} and the following lemma.
\end{proof}

\begin{lemma}\label{lemma:proba:1:mult:index}
Let $\Sigma$ be a subset of $(\ZZ_{\geqslant 1})^n$.
We assume that there is a positive integer $p_0$ with the following properties:
for any prime $P \in (\ZZ_{\geqslant 1})^n$ with $\min P \geqslant p_0$ (i.e., $P(i)$ is a prime number and $P(i) \geqslant p_0$ for every $i$),
there is $M' \in (\ZZ_{\geqslant 1})^n$ such that $P \cdot M \in \Sigma$ for any $M \geqslant M'$.
Then
\[
\lim_{m_1\to\infty, \ldots, m_n\to\infty}  \frac{ \#\{ I \in (\ZZ_{\geqslant 1})^n \mid \text{$I \in \Sigma$ and $I \leqslant (m_1, \ldots, m_n)$} \}}{m_1 \cdots m_n} = 1.
\]
\end{lemma}

\begin{proof}
Since the Riemann zeta function has a pole at $1$, for a positive number $\varepsilon$ with $0 < \varepsilon \leqslant 1$,
let $p_1, \ldots, p_{\ell}$ be prime numbers such that $p_0 \leqslant p_1 < \cdots < p_{\ell}$ and
$(1-1/p_1) \cdots (1-1/p_{\ell}) \leqslant \varepsilon/(4n)$.
For each $Q \in \{ p_1, \ldots, p_{\ell} \}^n$, by our assumption, there exists $M_Q \in (\ZZ_{\geqslant 1})^n$ such that
$Q \cdot M \in \Sigma$ for all $M \geqslant M_Q$.

\begin{claim}
If we set $e=p_1 \cdots p_{\ell}$, then we can find $M_0 \in \ZZ_{\geqslant 1}^n$ such that
\[
\{ I \in (\ZZ_{\geqslant 1})^n \mid \text{$I \geqslant M_0$ and $\GCD(I(i), e) \not= 1$ for all $i \in \{ 1, \ldots, n \}$} \} \subseteq \Sigma
\]
\end{claim}

\begin{proof}
We set $M_0 = \sum_{Q \in \{ p_1, \ldots, p_{\ell} \}^n} Q M_Q$. Let $I$ be an element  of $(\ZZ_{\geqslant 1})^n$ such that
$I \geqslant M_0$ and $\GCD(I(i), e) \not= 1$ for all $i \in \{ 1, \ldots, n \}$.
Then, for each $i \in \{1, \ldots, n \}$, there is $q_i \in \{  p_1, \ldots, p_{\ell} \}$ such that $q_i \mid I(i)$.
Thus, if we set $Q = (q_1, \ldots, q_n)$, then $I = Q \cdot M$ for some $M \in (\ZZ_{\geqslant 1})^n$. Moreover,
\[
QM_Q \leqslant M_0 \leqslant I = Q \cdot M,
\]
so $M \geqslant M_Q$, and hence $I \in \Sigma$.
\end{proof}

We set $M_0 = (a_1, \ldots, a_n)$. Then we have the following:

\begin{claim} If $m_1 \geqslant a_1, \ldots, m_n \geqslant a_n$, then
\[
\frac{ \# \{ I \in (\ZZ_{\geqslant 1})^n \setminus \Sigma \mid I \leqslant (m_1, \ldots, m_n) \}}{m_1 \cdots m_n}\leqslant  \sum_{i=1}^n \frac{a_i-1}{m_i} + (1-1/p_1)\cdots (1-1/p_{\ell})\sum_{i=1}^n \frac{m_i + e}{m_i} 
\]
\end{claim}

\begin{proof}
By the above claim,
\begin{align*}
(\ZZ_{\geqslant 1})^n \setminus  \Sigma & \subseteq \{ I \in (\ZZ_{\geqslant 1})^n \mid I \not\geqslant M_0\} \cup \{ I \in (\ZZ_{\geqslant 1})^n \mid 
\text{$\exists i \in \{ 1, \ldots, n \}$ $\GCD(I(i), e) = 1$}\} \\
& \kern-2em \subseteq \{ I \in (\ZZ_{\geqslant 1})^n \mid \text{$\exists i \in \{ 1, \ldots, n \}$ $I(i) < a_i$}\} \cup \bigcup_{i=1}^n \{ I \in (\ZZ_{\geqslant 1})^n \mid 
\text{$\GCD(I(i), e) = 1$}\} \\
 & \kern-2em \subseteq \bigcup_{i=1}^n \{ I \in (\ZZ_{\geqslant 1})^n \mid \text{$I(i) < a_i$} \} \cup \bigcup_{i=1}^n \{ I \in (\ZZ_{\geqslant 1})^n \mid 
\text{$\GCD(I(i), e) = 1$}\} 
\end{align*}
Thus, by \cite[the proof of Lemma~5.16]{IKMFaltings}, if $m_1 \geqslant a_1, \ldots, m_n \geqslant a_n$, then
{\allowdisplaybreaks\begin{multline*}
\frac{ \# \{ I \in (\ZZ_{\geqslant 1})^n \setminus \Sigma \mid I \leqslant (m_1, \ldots, m_n) \}}{m_1 \cdots m_n} \\
\kern-10em 
\leqslant  \sum_{i=1}^n 
\frac{\# \{ I \in (\ZZ_{\geqslant 1})^n \mid \text{$I(i) < a_i$ and $I \leqslant (m_1, \ldots, m_n)$} \}}{m_1 \cdots m_n} \\
\kern5em 
+ \sum_{i=1}^n \frac{\#\{ I \in (\ZZ_{\geqslant 1})^n \mid \text{$\GCD(I(i), e) = 1$ and $I \leqslant (m_1, \ldots, m_n)$} \}}{m_1 \cdots m_n} \\
\leqslant  \sum_{i=1}^n \frac{a_i-1}{m_i} + \sum_{i=1}^n \frac{(1-1/p_1)\cdots (1-1/p_{\ell})(m_i + e) }{m_i},
\end{multline*}}%
as required.
\end{proof}

We choose positive integers $b_1, \ldots, b_n$ such that $(a_i - 1)/b_i \leqslant \varepsilon/(2n)$ and $e/b_i \leqslant \varepsilon$.
Then, by the above claim, if $m_i \geqslant \max \{ a_i, b_i\}$ for all $i \in \{ 1, \ldots, n \}$, we have
\[
\frac{ \# \{ I \in (\ZZ_{\geqslant 1})^n \setminus \Sigma \mid I \leqslant (m_1, \ldots, m_n) \}}{m_1 \cdots m_n} \leqslant \varepsilon/2 + (\varepsilon/4) (1 + \varepsilon) \leqslant  \varepsilon,
\]
because $0 < \varepsilon \leqslant 1$, 
and hence the assertion follows.
\end{proof}

\begin{example}\label{example:P:P}
Let $f_{(A, A')}$ $(A, A'\geqslant 2$) be the endomorphism on $\PP^1 \times \PP^1 (= \Proj(K[X_0, X_1]) \times \Proj(K[Y_0, Y_1]))$
defined to be \[f_{(A, A')}((x_0 : x_1) \times (y_0 : y_1)) = (x_0^A : x_1^A) \times (y_0^{A'} : y_1^{A'}).\]
Let $p_i : \PP^1 \times \PP^1 \to \PP^1$ be the projection to the $i$-th factor, and $L = p_1^*(\OO_{\PP^1}(1)) \otimes p_2^*(\OO_{\PP^1}(1))$.
Note that $f_{(A, A')} \circ f_{(B, B')} = f_{(B, B')} \circ f_{(A, A')} = f_{(AB, A'B')}$ (multiplicative).
Let $Y$ be a subvariety of $\PP^1 \times \PP^1$ given by \[\alpha X_0Y_0+ \beta X_1 Y_1 + \gamma X_0 Y_1 + \delta X_1 Y_0 =  (X_0\ X_1)\Pi \begin{pmatrix} Y_0 \\ Y_1 \end{pmatrix}  = 0\quad(\alpha, \beta, \gamma, \delta \in K),\]
where $\Pi = \begin{pmatrix} \alpha & \gamma \\ \delta & \beta \end{pmatrix}$.
Then $Y_{(A,A')}$ is given by \[\alpha X_0^AY_0^{A'}+ \beta X_1^A Y_1 ^{A'}+ \gamma X_0 ^AY_1^{A'} + \delta X_1^A Y_0^{A'}  = (X^A_0 \ X^A_1 )\Pi\begin{pmatrix} Y^{A'}_0 \\ Y^{A'}_1 \end{pmatrix} = 0.\]
We assume that $\alpha\beta\gamma\delta \not= 0$ and $\det \Pi \not= 0$.
Then it is easy to see that $Y_{(A, A')}$ is smooth over $K$ for all $A,A' \geqslant 2$, so its genus is $(A-1)(A'-1)$.
Thus, Theorem~\ref{thm:prob:1:mult} holds in this case.
\end{example}

\bibliography{DynamicFermat}

@article{BMSCancellation,
    AUTHOR = {Bell, Jason and Matsuzawa, Yohsuke and Satriano, Matthew},
    TITLE = {On Dynamical Cancellation},
	JOURNAL = {International Mathematics Research Notices},
     YEAR = {2023},
	 PAGES = {7099-7139},
}

@book{MR1070709,
	author = {Berkovich, Vladimir G.},
	doi = {10.1090/surv/033},
	isbn = {0-8218-1534-2},
	mrclass = {32P05 (32C15 32C37 46S10 47S10)},
	mrnumber = {1070709},
	mrreviewer = {W. Bartenwerfer},
	pages = {x+169},
	publisher = {American Mathematical Society, Providence, RI},
	series = {Mathematical Surveys and Monographs},
	title = {Spectral theory and analytic geometry over non-{A}rchimedean fields},
	url = {https://doi.org/10.1090/surv/033},
	volume = {33},
	year = {1990}}

@book{silverman2007arithmetic,
  title={The Arithmetic of Dynamical Systems},
  author={Silverman, J.H.},
  isbn={9780387699035},
  lccn={2007923502},
  series={Graduate Texts in Mathematics},
  url={https://books.google.co.jp/books?id=oQ3KafF3rlMC},
  year={2007},
  publisher={Springer New York}
}

@incollection{MR1307396,
	author = {Faltings, Gerd},
	booktitle = {Barsotti {S}ymposium in {A}lgebraic {G}eometry ({A}bano {T}erme, 1991)},
	mrclass = {11G10 (14G05 14K99)},
	mrnumber = {1307396},
	mrreviewer = {Juan-Carlos Naranjo},
	pages = {175--182},
	publisher = {Academic Press, San Diego, CA},
	series = {Perspect. Math.},
	title = {The general case of {S}. {L}ang's conjecture},
	volume = {15},
	year = {1994}}

@article{MR734070,
	author = {Filaseta, Michael},
	fjournal = {La Soci\'{e}t\'{e} Royale du Canada. L'Academie des Sciences. Comptes Rendus Math\'{e}matiques. (Mathematical Reports)},
	issn = {0706-1994},
	journal = {C. R. Math. Rep. Acad. Sci. Canada},
	mrclass = {11D41},
	mrnumber = {734070},
	mrreviewer = {Takeo Miyoshi},
	number = {1},
	pages = {31--33},
	title = {An application of {F}altings' results to {F}ermat's last theorem},
	volume = {6},
	year = {1984}}

@article{MR766439,
	author = {Heath-Brown, D. R.},
	doi = {10.1112/blms/17.1.15},
	fjournal = {The Bulletin of the London Mathematical Society},
	issn = {0024-6093},
	journal = {Bull. London Math. Soc.},
	mrclass = {11D41 (11N37)},
	mrnumber = {766439},
	mrreviewer = {Takashi Agoh},
	number = {1},
	pages = {15--16},
	title = {Fermat's last theorem for ``almost all'' exponents},
	url = {https://doi.org/10.1112/blms/17.1.15},
	volume = {17},
	year = {1985}}

@article{MR777765,
	author = {Granville, Andrew},
	fjournal = {La Soci\'{e}t\'{e} Royale du Canada. L'Academie des Sciences. Comptes Rendus Math\'{e}matiques. (Mathematical Reports)},
	issn = {0706-1994},
	journal = {C. R. Math. Rep. Acad. Sci. Canada},
	mrclass = {11D41},
	mrnumber = {777765},
	mrreviewer = {David Lee Hilliker},
	number = {1},
	pages = {55--60},
	title = {The set of exponents, for which {F}ermat's last theorem is true, has density one},
	volume = {7},
	year = {1985}}

@book{IKMFaltings,
	author = {Ikoma, Hideki and Kawaguchi, Shu and Moriwaki, Atsushi},
	doi = {10.1017/9781	108991445},
	isbn = {978-1-108-84595-3},
	pages = {vii+169},
	publisher = {Cambridge University Press, Cambridge},
	series = {Cambridge Tracts in Mathematics},
	title = {The {M}ordell Conjecture: a complete proof from {D}iophantine geometry},
	volume = {226},
	year = {2022}}

@book{MR4292529,
	author = {Chen, Huayi and Moriwaki, Atsushi},
	doi = {10.1007/978-981-15-1728-0},
	isbn = {978-981-15-1727-3; 978-981-15-1728-0},
	mrclass = {14G40 (11 37)},
	mrnumber = {4292529},
	pages = {450},
	publisher = {Springer, Singapore},
	series = {Lecture Notes in Mathematics},
	title = {Arakelov geometry over adelic curves},
	url = {https://doi.org/10.1007/978-981-15-1728-0},
	volume = {2258},
	year = {[2020] \copyright 2020}}

@book{CMIntersection,
	author = {Chen, Huayi and Moriwaki, Atsushi},
	isbn = {978-981-15-1727-3; 978-981-15-1728-0},
	mrclass = {14G40, 14G25, 11G50, 11G35, 37P30},
	note = {to appear in M\'emoires de la Soci\'et\'e math\'ematique de France},
	pages = {vi+147},
	title = {Arithmetic intersection theory over adelic curves},
	year = {2022}}

@book{Chen_Moriwaki24,
	author = {Chen, Huayi and Moriwaki, Atsushi},
	pages = {vii+247},
	publisher = {Birkh\"user, Switzerland},
	series = {Progress in Mathematics},
	title = {Positivity in Arakelov Geometry over Adelic Curves},
	volume = {355},
	year = {2024}}
\bibliographystyle{plain}

\end{document}